\documentclass[11pt,reqno]{amsart}

\usepackage{enumerate,amsmath,amssymb,amsthm}
\usepackage{verbatim}
\usepackage[usenames,dvipsnames]{pstricks}
\usepackage{epsfig}
\usepackage{pst-grad}
\usepackage{pst-plot}
\usepackage{epstopdf}
\usepackage{float}
\usepackage{geometry,graphicx,fancybox}
\usepackage{tikz}
\usepackage{caption}
\theoremstyle{plain}

\newtheorem{maintheorem}{Theorem}

\newtheorem{theorem}{Theorem}[section]
\newtheorem{corollary}[theorem]{Corollary}
\newtheorem{lemma}[theorem]{Lemma}

\newtheorem{proposition}[theorem]{Proposition}

\theoremstyle{remark}

\theoremstyle{definition}
\newtheorem{definition}[theorem]{Definition}
\newtheorem{remark}[theorem]{Remark}

\newcommand{\Z}{\mathbb{Z}}
\newcommand{\R}{\mathbb{R}}
\newcommand{\N}{\mathbb{N}}

\newcommand{\ie}{{\it i.e.}}

\DeclareMathOperator{\area}{Area}
\DeclareMathOperator{\length}{length}

\DeclareMathOperator{\sys}{sys}

\DeclareMathOperator{\dist}{dist}

\DeclareMathOperator{\card}{card}
\DeclareMathOperator{\Cut}{Cut}
\DeclareMathOperator{\expo}{exp}

\numberwithin{equation}{section}

\usepackage{mathtools}
\DeclarePairedDelimiter{\ceil}{\lceil}{\rceil}
\DeclarePairedDelimiter{\floor}{\lfloor}{\rfloor}

\makeatletter
\newcommand{\mypm}{\mathbin{\mathpalette\@mypm\relax}}
\newcommand{\@mypm}[2]{\ooalign{%
  \raisebox{.1\height}{$#1+$}\cr
  \smash{\raisebox{-.6\height}{$#1-$}}\cr}}
\makeatother

\title{Short homotopically independent loops on surfaces.}
\author{Steve Karam}
\begin{document}

\begin{abstract}
In this paper, we are interested in short  
homologically and homotopically
independent loops based at the same point
on Riemannian surfaces and metric graphs.

First, we show that for every closed
Riemannian surface of
genus~$g~\geq~2$ and area  normalized to~$g$, there are
at least~$\ceil{\log(2g)+1}$ homotopically
independent loops based at the same point
of length at most~$C\log(g)$, where~$C$ is a universal constant.
On the one hand, this result substantially
improves Theorem~$5.4.A~$ of M.~Gromov in \cite{G1}.
On the other hand, it recaptures the result of S.~Sabourau on the
separating systole in \cite{SS} and  refines his proof.

Second, we show that for any two integers~$b\geq 2$ with~$1\leq n\leq b$, every
connected metric graph~$\Gamma$ of first Betti number~$b$ and
of length~$b$ contains at least~$n$ homologically
independent loops based at the same point and
of length at most~$24(\log(b)+n)$.
In particular, this result extends
Bollob\`as-Szemer\'edi-Thomason's~$\log(b)$ bound on the
homological systole to at least~$\log(b)$ homologically
independent loops based at the same point.
Moreover, we give examples of graphs where this result is optimal.

\end{abstract}

\address{Steve Karam, Laboratoire de Math\'ematiques et de Physique Th\'eorique,
UFR Sciences et Technologie, Universit\'e Fran\c cois Rabelais, Parc de Grandmont,
37200 Tours, France} \email{steve.karam@lmpt.univ-tours.fr}

\maketitle

\section{Introduction}\label{section.introduction}
Short homotopically and homologically  independent
loops on surfaces have been of a great interest.
Gromov proved in~\cite{G1} and~\cite{G2}
that both~$\sys(M)$, the systole,
i.e., the shortest non-contractible loop,
and~$\sys_H(M)$, the homological systole, 
i.e., the
shortest homologically nontrivial loop,
of a closed Riemannian surface~$M$ of
genus~$g\geq~2$  with area normalized
to~$4\pi(g-1)$ are at most~$\sim\log(g)$.
In~\cite{BPS}, F.~Balacheff, S.~Sabourau and H.~Parlier
found the maximal number of homologically independent loops 
of length at most~$\sim\log(g)$.
Their theorem goes as follows.

\begin{theorem}[\cite{BPS}] \label{thmBPS1}
Let~$\eta:\N \to \N$ be a function such that
$$
\displaystyle \lambda := \sup_g \frac{\eta(g)}{g} < 1.
$$
Then there exists a constant~$C_\lambda$ such that for every
closed Riemannian surface~$M$ of genus~$g$ there
are at least~$\eta(g)$ homologically
independent loops~$\alpha_{1},\ldots,\alpha_{\eta(g)}$ which satisfy
\begin{equation*}
\length(\alpha_{i}) \leq C_\lambda \, \frac{\log(g+1)}{\sqrt{g}} \, \sqrt{\area(M)},
\end{equation*}
for every~$i \in \{1,\ldots,\eta(g)\}$.
\end{theorem}
\noindent Moreover, they constructed some hyperbolic surfaces
where their bound is optimal.

\bigskip

For the applications we have in mind (see Section~\ref{section.benefits}),
it would be nice if the loops in Theorem~\ref{thmBPS1} were based at the same point.
Unfortunately, the following example shows that in general,
we cannot  even find two homologically independent loops 
based at the same point satisfying a $\log(g)$ bound.
Indeed, let $M$ be a closed hyperbolic surface of genus $g$. 
Consider a family of $g+1$ loops in~$M$
dividing the surface into two spheres with $g+1$ boundary components.
Pinching these loops enough, we force (by the collar theorem)
every loop of $M$ homologically independent from this family 
to be arbitrary long.
Still, we obtain some result in this direction when the systole
is bounded from below, see Theorem~\ref{maintheorem1.with.lower.bound}.
\\

This leads us to replace the notion of homologically independent 
loops by the notion of homotopically independent loops defined below.

\begin{definition}
Let~$M$ be a closed Riemannian surface of genus at least one.
A family of loops~$(\alpha_1,...\alpha_k)$ 
based at the same point~$v$ in~$M$
are said to be \emph{homotopically independent} if
the subgroup of~$\pi_1(M,v)$ generated by~$\alpha_1,\ldots,\alpha_k$
is free of rank~$k$.
\end{definition}

\noindent Observe that~$k$ homologically independent loops based at the same 
point on a closed surface~$M$ of genus~$g$ are homotopically 
independent for~$k<2g$, see Theorem~\ref{jaco}.

\bigskip

Now we ask the following question: for how many  homotopically  independent loops based
at the same point does the~$\log(g)$ bound hold?
\\

\noindent One might wonder or even doubt 
the benefit of finding short homotopically independent 
loops based at the same point. We show the
benefits of such a choice in Section \ref{section.benefits}.
To the author best knowledge, the only
answer to the previous question is due to Gromov.

\begin{theorem}[\cite{G2}, 5.4.B]\label{thmGromov}
Let~$(M,h)$ be a closed Riemannian surface of
genus~$g\geq 2$ and of area normalized to~$g$.
For every~$\alpha<1$, there exist two
homotopically independent loops~$\gamma_1$ and~$\gamma_2$
based at the same point in~$M$ such that
$$sup(\length(\gamma_1),\length(\gamma_2))\leq C_\alpha\;g^{1-\alpha},$$
where~$C_{\alpha}$ is a positive constant that depends only on~$\alpha$.
\end{theorem}

\noindent

Note that Theorem~\ref{thmGromov} does not hold for~$\alpha=1$.
Indeed, P. Buser and P. Sarnak constructed in \cite{BS94} hyperbolic surfaces
with injectivity radius~$\sim\log(g)$ at every point.
We improve Theorem \ref{thmGromov} by showing the following result.

\medskip

Throughout this paper for a positive real number~$R$,
we denote by~$\ceil{R}$
the smallest integer greater or equal to~$R$.

\begin{maintheorem}\label{maintheorem1}
Let~$M$ be a closed Riemannian surface of genus~$g\geq 2$.
There are at least~$\ceil{\log(2g)+1}$
homotopically independent loops~$\alpha_1,\ldots ,\alpha_{\ceil{\log(2g)+1}}$
based at the same point in~$M$, such that
for every~$i\in \{1,...,\ceil{\log(2g)+1} \}$,
$$\length(\alpha_i)\leq C\frac{\log(g)}{\sqrt{g}}\sqrt{\area(M)},$$
where~$C$ is a universal constant independent from the genus.
\end{maintheorem}

Theorem \ref{maintheorem1} substantially improves Theorem~\ref{thmGromov}.
Under the same hypothesis as Theorem~\ref{thmGromov},
Theorem~\ref{maintheorem1} guarantees the existence
of~$\ceil{\log(2g)+1}$  homotopically independent loops
based at  the same point (instead of two)  of length roughly
bounded by~$\log(g)$ (instead of~$g^{\alpha}$).
Note that, if the homotopical systole of
the surface~$M$ in Theorem~\ref{maintheorem1}
is bounded away from zero,
then the~$\ceil{\log(2g)+1}$ loops can be even 
chosed to be  homologically independent
(see Theorem~\ref{maintheorem1.with.lower.bound}).
Also Theorem \ref{maintheorem1} recaptures
the following result  by S.~Sabourau.

\begin{theorem}[Sabourau, \cite{SS}]\label{theorem.sabourau.separating systole}
There exists a positive constant~$C$ such that every closed Riemannian surface~$M$
of genus~$g\geq 2$ and area normalized to~$g$, satisfies
$$\sys_0(M) \leq C \log(g),$$
where~$\sys_0(M)$ is denifed as the length of the shortest non-contractible
loop in~$M$ which is trivial in~$H_1(M,\Z)$.
\end{theorem}

\noindent Note that  Sabourau  splits his proof into two cases.
In the first case, he supposes
\linebreak
that~$\sys_0(M)\leq 4 \sys(M)$
and then he deduces the result from Gromov's~$\log(g)$  bound on the systole.
Meanwhile, Theorem \ref{maintheorem1} provides a unified proof of this theorem
without refering to Gromov's asymptotic systolic inequality.

\medskip

Gromov's~$\log(g)$ bound on the systole has an  analog for
metric graphs. Note that for a metric graph~$\Gamma$, the homotopical
systole coincides with the homological systole.
We will  denote it  by~$\sys(\Gamma)$.
The best bound on the  systole of a metric graph is due to
B. Bollob\`as, E. Szemer\'edi and B. Thomason \cite{BS02}, \cite{BT}.
Specifically, they proved that
the systole of every connected metric graph
of first Betti number~$b\geq 2$, and length normalized to~$b$
is at most $4\log(b+1)$.

\medskip

Exactly as for surfaces, given a metric graph of first Betti
number~$b~\geq~2$ and of length normalized to~$b$,
one might wonder about the number  of
homologically independent loops based at the same point satisfying the
B.~Bollob\`as, E.~Szemer\'edi and B.~Thomason~$\log(b)$ bound.
We answer this question here.

\begin{maintheorem}\label{maintheorem2}
Let~$\Gamma$ be a connected  metric graph of first Betti
number~$b\geq 2$ and of length normalized~to~$b$. Let~$n\in \{1,\ldots ,b\}$.
There exist at least~$n$ homologically independent loops in~$\Gamma$ based at the same point
and of length at most~$24(\log(b)+n)$.
\end{maintheorem}

\noindent An interesting value of~$n$ is~$n=\floor{\log(b)}$,
$\ie$, the integral part of $\log(b)$.
In this case, Theorem~\ref{maintheorem2} asserts that
for every  connected metric graph~$\Gamma$
of first Betti number~$b~\geq~2$ and of length~$b$, there exist
at least~$\floor{\log(b)}$ homologically independent
loops based at the same point  of length at most~$48\log(b)$.
This extends B.~Bollob\`as, E.~Szemer\'edi and B.~Thomason~$\log(b)$
bound on the homological systole of~$\Gamma$
to~$\floor{\log(b)}$ homologically independent
loops of~$\Gamma$ based at the same point.
\\

One might wonder how far from being optimal Theorem \ref{maintheorem2} is.
We show that it cannot be substantially improved.
Indeed, let~$b$ and~$n$ be two integers such that~$b\geq 2$ and~$1\leq n\leq b$.
There exists a connected metric
graph of first Betti number~$b$ and length normalized to~$b$,
such that there are at most~$\floor{24(\log(b)+n)}+1$
homologically independent loops in~$\Gamma$ based at the same point
of length at most~$24(\log(b)+n)$ (\emph{cf.} Theorem \ref{maintheorem2.optimal}).
In particular, this result shows that for~$n\geq \ceil{\log(b)}$,
there exists a connected metric graph~$\Gamma$ of first Betti number
and length normalized to~$b$, such that there are at most~$52n$
homologically independent loops in~$\Gamma$ based at the same point
of length at most~$24(\log(b)+n)$.

\bigskip

This paper is organised as follows.
In Section \ref{section.benefits}, we show the benefits of short
homotopically independent loops based at the same point.
In Section \ref{section.proof.maintheorem2}, we give the proof of Theorem \ref{maintheorem2}.
In Section \ref{section.proof.maintheorem1.with.lower.bound}, we show how to extend
Theorem \ref{maintheorem2} to closed surfaces with systole bounded away from zero.
In Section \ref{section.cut.loci}, we show that on a given closed surface the cut locus
of a simple closed geodesic  captures its topology.
In Section \ref{section.proof.maintheorem1}, we prove Theorem \ref{maintheorem1}.
\\

\textbf{Acknowledgment}.
The author would like to thank his advisor, St\'ephane Sabourau,
for many useful discussions and valuable comments. 
He also would like to thank Florent Balacheff for reading and commenting this paper.

\section{Benefits of short homotopically independent loops based at the same point}\label{section.benefits}
In this section, we show two applications of homotopically independent loops based at the same point
of bounded length.
\\

Let~$M$ be a closed Riemannian surface  of genus~$g\geq 2$.
If~$\alpha$ and~$\beta$ are two homotopically
independent loops based at the same point in~$M$,
then~$$\sys_0(M)\leq \length(\alpha\beta\alpha^{-1}\beta^{-1}).$$
\\
In particular, if~$\sup(\length(\alpha),\length(\beta))\leq C\log(g)$,
then
$$\sys_0(M)\leq 4C\log(g).$$
Notice that the above observation allows us to recapture
the result of Theorem~\ref{theorem.sabourau.separating systole} on the separating systole
by means of Theorem~\ref{maintheorem1}.
Also we would like to point out that Gromov's upper bound~$C_\alpha\;g^{1-\alpha}$ on the length
of  two homotopically independent loops based
at the same point in Theorem~\ref{thmGromov} is not sufficient to prove
that the length of the separating systole of a closed Riemannian surface
of genus~$g\geq 2$ and area~$g$ is bounded above by~$\sim\log(g)$.

\bigskip

Another use of homotopically  independent loops based at the same point~$v$
of a closed Riemannian surface~$M$, is to contribute
to the area of balls centered at a lift~$\widetilde{v}$ of~$v$ in
the universal cover~$\widetilde{M}$ of~$M$.
Let us clarify this idea here.
Consider a system~$S=\{ \alpha_1,~\ldots ,~\alpha_k \}$ of pairwise
non-homotopic loops based at~$v$. Let
$$L=\sup_{1\leq i \leq k}\length(\alpha_i).$$
Denote by~$s$  half  the systole of~$M$ at the point~$v$,~$\ie$ half
the length of the shortest non contractible loop
based at~$v$. Let~$H'_r$ (resp~$N_r$) be
the set of elements of~$H=\langle S \rangle$
(resp~$\pi_1(M,v)$) of length less than~$r$,
where the length of~$\alpha\in\pi_1(M,v)$ is defined
as~$\length(\alpha)=\dist(\widetilde{v},\alpha.\widetilde{v})$.
It is the minimal length of a loop based at~$v$ representing~$\alpha$.
Let~$R>s+L$. Consider the ball~$B=B_{\widetilde{M}}(\widetilde{v},r_0)$,
where~$r_0=R-s$.
Every element~$\gamma_i$ of~$N_{r_0}$ yields a
point~$\widetilde{v}_i=\gamma_i.\widetilde{v}$ in~$B$.
The balls~$B_{\widetilde{M}}(\widetilde{v}_i,s)$
are disjoint and of the same area. We have
\begin{eqnarray}
\area B_{\widetilde{M}}(\widetilde{v},R)\geq \card(N_{r_0}) \area B_M(v,s),
\end{eqnarray}
where~$\card(N_{r_0})$ is the cardinal of~$N_{r_0}$.
\\
Also notice that
\begin{eqnarray}
\card(N_{r_0})\geq \card(H'_{r_0}).
\end{eqnarray}
Thus, a lower bound on the cardinal~$\card(H'_{r_0})$
of~$H'_{r_0}$ yields also a lower bound on~$\card(N_{r_0})$.
One way to bound~$\card(H'_{r_0})$ from below is the following.
We define a norm~$\parallel \;\; \parallel$ on~$H$ as follows.
For~$\beta$ in~$H$, we define the word length~$\parallel \beta \parallel$
of~$\beta$ as the smallest integer~$n$ such
that~$\beta=\gamma_1 \ldots \gamma_n$
where~$\gamma_i \in S\cup S^{-1}$.
Denote by~$H_{r}^w$ the set of elements of~$H$ of word length less than~$r$.
We have
\begin{eqnarray}
\card(N'_r)\geq \card(H_{r/L}^w).
\end{eqnarray}
Combining~$(3.1)$,~$(3.2)$ and~$(3.3)$ we got
\begin{eqnarray}
\area B_{\widetilde{M}}(\widetilde{v},R)\geq \card(H_{r/L}^w) \area B_M(v,s).
\end{eqnarray}
Now let~$r'>1$. Notice that~$H_{r'}^w$ is maximal
if~$H$ is free of rank~$k$. That is guaranteed if
the loops~$\alpha_1,\ldots \alpha_k$
are homotopically independent in~$M$.
It is now clear how homotopically
independent loops based at the same point~$v$
contribute to the area to the balls centered at  points in the
fiber over~$v$ in~$\widetilde{M}$ whenever the radii~$R$ of these balls
is longer than~$s+L$. Moreover, since~$R$ must be at least~$s+L$,
it is straightforward to see that  the shorter the~$L$,  the better the result.
This means that the upper bound of the lengths of the~$\alpha_i$'s is also important.

\section{Short homologically  independent loops on  graphs}\label{section.proof.maintheorem2}

In this section we prove Theorem \ref{maintheorem2}.
Recall that this theorem extends the Bollob\`as-Szemer\'edi-Thomason~$\log(b)$ bound
on the homological systole of graphs to~$\ceil {\log(g)}$ homologically
independent loops based at the same point.

\bigskip

First let us recall  some definitions.
By definition, a graph $\Gamma$ is a  finite one-dimensional CW-complex
(multiple edges and loops are allowed).
The first Betti number  of a graph~$\Gamma$ can be computed as follows:
$$b(\Gamma)=e-v+n,$$
where~$e,v$ and~$n$ are respectively the number of edges,
vertices and connected components of~$\Gamma$.
A metric graph~$(\Gamma,h)$ is a graph endowed with a length space metric~$h$.
The length of a subgraph of~$\Gamma$ is its one-dimensional Hausdorff measure.
For more details on graphs we refer the reader to \cite{DGT}.

\begin{definition}
Let~$\Gamma$ be a connected graph of first Betti number $b\geq 1$.
A family of loops~$(\alpha_1,...\alpha_k)$ in~$\Gamma$
is said to be \emph{homologically independent}
if their homology classes in~$H_1(\Gamma,\R)$ are.
\end{definition}

\noindent Note that this definition extends also to closed Riemannian manifolds.

\bigskip

\noindent Now we prove Theorem~\ref{maintheorem2}.

\begin{theorem}\label{maintheorem2.bis}
Let~$\Gamma$ be a connected  metric graph of first Betti
number~$b\geq 2$ and of length normalized~to~$b$. Let~$n\in \{1,\ldots ,b\}$.
There exist at least~$n$ homologically independent loops in~$\Gamma$ based at the same point
and of length at most~$24(\log(b)+n)$.
\end{theorem}

\begin{proof}
By definition of the first Betti number~$b$,
there exist~$b$ homologically independent
loops~$\alpha_1,\ldots,\alpha_b$ in~$\Gamma$.
Fix a point~$x$ of~$\alpha_1$.
For~$i=1,\ldots,b$, let~$C_i$ be a minimizing curve from~$x$ to~$\alpha_i$.
We have~$\length(C_i\alpha_i {C_i}^{-1})  \leq \length(C_i)+\length(\alpha_i)+\length(C_i)$ .
Notice that  $\length(C_i)+\length(\alpha_i) \leq b$.
Thus, there exists~$b$
homologically independent loops in~$\Gamma$ based
at the same point  of length at most~$2b$
($\leq 24(\log(b)+\frac{b}{2})$).
This yields the desired result for~$n\in\{\frac{b}{2},b\}$.
Now we consider the case when~$n<\frac{b}{2}$.
In particular, we suppose~$b\geq 3$.
By a short cycle of~$\Gamma$ we
mean a simple loop of length at most~$12\log(b)$.
Let~$X$  be a maximal set of
homologically independent short cycles
of~$\Gamma$ and denote by~$N$ its cardinal.
We claim that
$$N\geq \frac{b}{2}.$$
Indeed, we construct~$k=\ceil{\frac{b}{2}}$
graphs~$\Gamma_{k} \subset \ldots \subset \Gamma_1=\Gamma$ and~$k$ simple loops as follows.
Remove an edge from a systolic loop~$\gamma_1$
of~$\Gamma_1$ and denote by~$\Gamma_2$ the resulting graph.
The graph~$\Gamma_2$ is connected and of first
Betti number~$b_2~=~b-1$.
Now remove an edge from a systolic loop~$\gamma_2$ of~$\Gamma_2$
and denote by~$\Gamma_3$ the resulting graph.
By induction, we keep doing this until we get~$\Gamma_k$.
From the inequality~$(1.1)$ and since~$k=\ceil{\frac{b}{2}}$
we have for every~$i=1,\ldots, k$,
\begin{eqnarray}
\length(\gamma_i) &\leq& 4\frac{\log(1+b-i+1)}{b-i+1}\length(\Gamma_i) \nonumber \\
                  &\leq & 12\log(b). \nonumber
\end{eqnarray}
By construction, the~$k$ loops~$\{\gamma_i \}_{i=1}^k$
are homologically independent in~$\Gamma$. So the claim is proved.
\\
We divide the set~$X$ as follows.
Take any element~$\alpha_1$ of~$X$  and denote by~$Y_1$
the set~$\{\beta\in X \; | \; \dist(\beta ,\alpha_1)~\leq~4n \}$.
Let~$\alpha_2$ be an element of~$X \setminus Y_1$
and denote by~$ Y_2$ the set~$\{\beta\in X \; | \; \dist(\beta ,\alpha_2)~\leq~4n \}$.
By induction we continue  this process which eventually ends since~$X$ is finite.
Let~$\alpha_j\in X$ be the last short cycle obtained from this process,~$\ie$,
let~$\alpha_j$ be an element
of~$X~\setminus~Y_1~\cup~...~\cup~Y_{j-1}~$
such that~$Y_1 \cup \ldots \cup Y_{j-1}\cup Y_{j}=X$.
For~$i=1,\ldots,j$, we denote by~$T_i$ the cardinal of~$Y_i$.
We claim that there exists an~$i_0$ such that
$$T_{i_0}\geq n.$$
Indeed, suppose the opposite.
We have
$$\frac{b}{2}\leq N=\card(X)\leq \displaystyle\sum\limits_{i=1}^{j}T_i < jn .$$
So~$j>\frac{b}{2n}>1$. For~$i\neq i'$, we have
$\dist (\alpha_i,\alpha_{i'})>4n$. This means that the
open neighborhoods of radius~$2n$ around
the~$\alpha_i's$ are pairwise disjoint.
Since~$\Gamma$ is connected, the length of the
neighborhood of radius~$2n$ around each
short cycle~$\alpha_i$ is at least~$\length(\alpha_i)+2n$.
This implies that
$$\length(\Gamma)> 2nj > b.$$
Hence a contradiction.
So there is an~$i_0$ such that~$T_{i_0} \geq n$.
\\
Now fix a vertex~$a$ of~$\alpha_{i_0}$ and  let~$\beta$
be any element of~$Y_{i_0}\setminus \{\alpha_{i_0}\}$.
Let~$b$ and~$c$ be two vertices of~$\alpha_{i_0}$ and~$\beta$
respectively such that~$\dist(\alpha_{i_0},\beta)=~\dist(b,c)$.
Also, let~$C_{ab}$ be a minimizing curve from~$a$ to~$b$ and
$C_{bc}$ be a minimizing curve from~$b$ to~$c$.
The following holds.
\begin{itemize}
\item~$\length(C_{ab})~\leq~\length(\alpha_{i_0})/2$
\item~$\length(C_{bc})~\leq~4n$.
\end{itemize}
The loop~$\beta'=C_{ab}C_{bc}\beta C_{cb}C_{ba}$
is homologuous to~$\beta$ and satisfies
$$\length(\beta')\leq 24\log(b)+8n.$$
So the~$T_{i_0}$ short cycles of~$Y_{i_0}$ give rise to~$T_{i_0}$
homologically independent loops of~$\Gamma$ based at the
same point~$a$ and of length at most~$24(\log(b)+n)$.
\end{proof}

\begin{corollary}\label{corollary.maintheorem2.bis}
Let~$\Gamma$ be a connected metric graph of
first Betti number~$b\geq~2$. Let~$n\in\{1,\ldots,b\}$.
There exist at least~$n$ homologically independent
loops in~$\Gamma$ based at the same point
of length at most~$24(\log(b)+~n)~\frac{\length(\Gamma)}{b}$.
\end{corollary}

\noindent Before stating our next theorem, we construct a connected metric
graph~$\Gamma_\star$ that will be useful to the rest of this section.
Let~$m$ and~$p$ be two positive integers with~$m \geq p$.
Denote by~$q$ and~$r$  the quotient and the remainder
in the division of~$m$ by~$p$,
that is,~$m=pq+r$ with~$r\in\{0,\ldots,p-1\}$.
Also let~$L$ and~$l$ be two positive constants.
\\
Fix a vertex~$v$.
We construct~$q$ bouquets~$X_1,\ldots,X_q$ of~$p$ circles
and a bouquet~$X_{q+1}$ of~$r$ circles.
We define~$\Gamma_\star$ by joining
the vertex of each bouquet~$X_i$
to the vertex~$v$ by an edge~$w_i$. See Figure~1.
We define a metric~$h$ on~$\Gamma_\star$ such that~$(\Gamma_\star,h)$ is
a length metric space as follows.
For~$i=1,\ldots,q$, set~$\length(w_i)=~L$, and~$\length (X_i)=~l$.
Also set~$\length(X_{q+1})+ \length(w_{q+1})=r$.
It is straightforward to see that
the graph~$\Gamma_\star$ is connected,
of first Betti number~$m$ and of length~$q(L+l)+r$.
We claim that  there are at most~$p+r$~$(\leq 2p-1)$
homologically independent loops  based
at the same point of length at most~$2L$.
Indeed, notice that there exist at
most~$r$ homologically independent loops based at~$v$
of length less than~$2L$. So let~$m$ be any point
of~$\Gamma_\star$ other than the point~$v$.
There exists a unique~$i$
such that~$m\in X_i\cup w_i$. Now notice that if we want to find more
than~$p+r$ homologically independent loops based at~$m$,
one of them must cross at least two times one of the
edges~$w_j$, with~$j\in \{1,\ldots, q \}\setminus \{i\}$.
Thus, the length of this loop exceeds~$2L$.

\begin{figure}[H]
\begin{center}
\includegraphics[height=45mm]{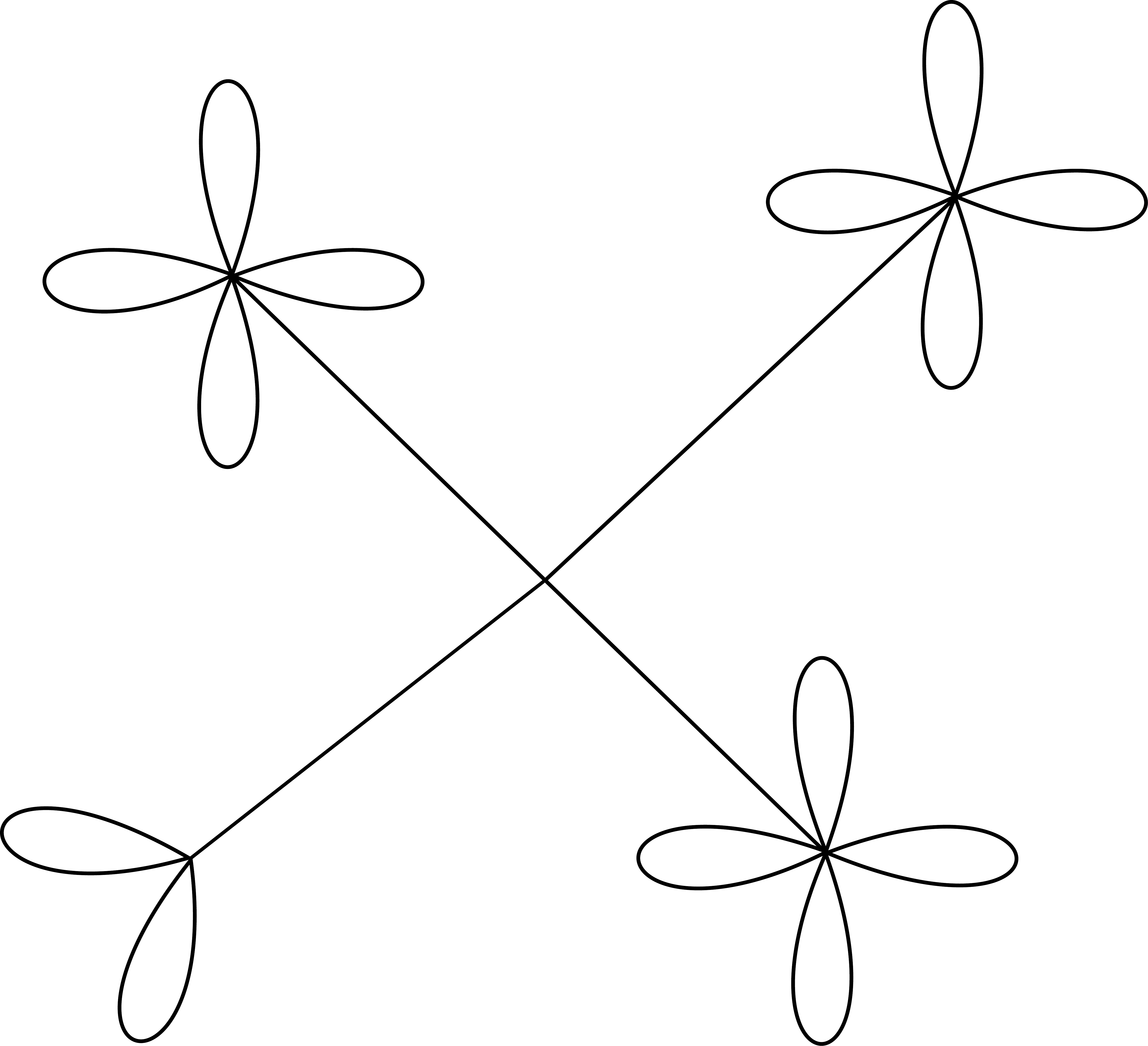}
\leavevmode
\put(2,113){\makebox(0,0){$X_i$}}
\put(-53,55){\makebox(0,60){$w_i$}}
\put(-74,64){\makebox(0,0){$v$}}
\put(-82.5,62){\makebox(-30,20){$L$}}
\put(-82,74){\makebox(-40,70){$l$}}
\caption{The graph~$\Gamma_\star$ for~$m=12$,~$p=4$,~$q=3$ and~$r=2$.}
\end{center}
\end{figure}
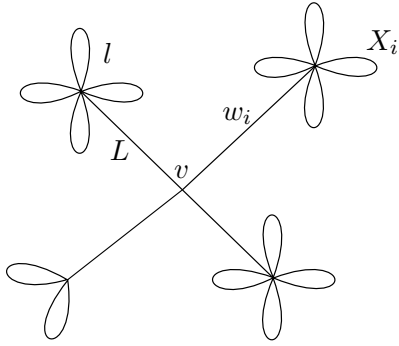

Our next theorem shows that one cannot substantially
improve Theorem~\ref{maintheorem2.bis}, thus it is roughly optimal.

\begin{theorem}\label{maintheorem2.optimal}
Let~$b$ and~$n$ be two integers such that~$b\geq 2$ and~$1\leq n\leq b$.
Let~$\lambda>0$. There exists a connected metric
graph of first Betti number~$b$, of length normalized to~$b$,
such that there are at most~$\floor{\lambda(\log(b)+n)}+1$
homologically independent loops in~$\Gamma$ based at the same point
of length at most~$\lambda(\log(b)+n)$
\end{theorem}

\begin{proof}
We only need to consider the case
when~$b\geq \floor{\lambda(\log(b)+n)}+1$ since the other case is trivial.
Denote by~$q$ and~$r$  respectively the quotient and the remainder
in the division of~$b$ by~$\floor{\frac{\lambda}{2}(\log(b)+n)}+1$.
Let~$\varepsilon>0$ be such that
$$\floor{\frac{\lambda}{2}(\log(b)+n)}+1= \frac{\lambda}{2}(\log(b)+n)+\varepsilon.$$
\noindent Consider  the graph~$\Gamma_\star$ given by the previous construction
with
\begin{itemize}
\item $m=b$,
\item $p=\floor{\frac{\lambda}{2}(\log(b)+n)}+1$,
\item $L=\frac{\lambda}{2}(\log(b)+n)$,
\item $l=\varepsilon$.
\end{itemize}
The graph~$\Gamma_\star$ is connected,
of first Betti number~$b$, of length~$b$ and has at
\linebreak
most~$\floor{\lambda(\log(b)+n)}+1$ homologically independent loops
based at the same point  of length at most~$\lambda(\log(b)+n)$.
\end{proof}

\section{Short homologically independent loops on surfaces with homotopical systole bounded from below.}\label{section.proof.maintheorem1.with.lower.bound}

In this section we combine ideas from \cite{BPS} and \cite{Ka} to
extend Theorem~\ref{maintheorem2.bis} 
to closed surfaces with systole bounded below.

\begin{definition}
Let~$(M,h)$ be a  closed Riemannian surface of genus~$g$.
The image in~$M$ of an abstract graph by an
embedding will be referred to as a graph in~$M$.
The metric~$h$ on~$M$ naturally induces a metric
on a graph~$\Gamma$ in~$M$. Despite the risk of confusion,
we will also denote by~$h$ such a metric on~$\Gamma$.
\end{definition}

\begin{proposition}\label{lengthofminimalgraph}
Let~$(M,h)$ be a closed  Riemannian surface of genus~$g~\geq~1$.
Suppose that the homotopical systole of~$M$ is at least~$\ell$.
Then, there exists a graph~$\Gamma$ in~$M$ such that
\begin{enumerate}
\item the inclusion map~$i:\Gamma\to M$ is distance non-increasing;
\item the homomorphism~$i_{\ast}:H_1(\Gamma,\R) \to H_1(M,\R)$
induced by the inclusion is an isomorphism;
\item $$\length(\Gamma)\leq \frac{2^9 \area(M,h)+g}{\min\{1,\ell\}}.$$
\end{enumerate}
\end{proposition}

\begin{proof}
Without loss of generality, we suppose that~$\ell \leq 1$.
This proposition is the same as Proposition~$6.1$
in \cite{Ka}, where~$\ell$ was taken to be~$\frac{1}{2}$ and the area
is equal to~$\frac{1}{2^{11}}(2g-1)$ instead of~$g$.
The proof of Proposition 6.1 in \cite{Ka} starts by fixing~$r_0=\frac{1}{2^5}$,
In our case we fix~$r_0=\frac{\ell}{2^4}$ and reproduce the argument.
\end{proof}

Before stating out next theorem, let us recall the following theorem.

\begin{theorem}[\cite{jaco}]\label{jaco}
Let~$M$ be a closed Riemann surface of Euler characteristic~$\chi(M)\leq 0$.
Any subgroup of~$\pi_1(M)$ generated by~$k$ elements, where~$k<2-\chi(M)$,
is a free group.
\end{theorem}

\noindent Now we can prove the following result.

\begin{theorem}\label{maintheorem1.with.lower.bound}
Let~$M$ be a closed orientable Riemannian surface
of genus~$g\geq~1$ with homotopical systole
at least~$\ell$ and area normalized to~$g$.
Let~$n\in \{1,\ldots,2g \}$ be an integer.
There exist at least~$n$ homologically independent
loops~$\gamma_1,...,\gamma_n$
based at the same point in~$M$ such that
for every~$i=1,\ldots,n$, we have
$$\length(\gamma_i)\leq 24C_{\ell}(\log(2g)+n),$$
where~$C_{\ell}=\frac{2^{9}}{\min\{1,\ell\}}$.
\\
Moreover, if~$n<2g$ then~$\langle \gamma_1,...,\gamma_n \rangle$ is free of rank~$n$.
\end{theorem}

\begin{proof}
Let~$\Gamma$ be a graph in~$M$ that satisfies
(1), (2) and (3) of Proposition \ref{lengthofminimalgraph}.
The first Betti number of~$\Gamma$ is~$2g$.
By Corollary \ref{corollary.maintheorem2.bis},
there are at least~$n$ homologically independent loops in~$\Gamma$ based
at the same vertex of length at most~$24C_{\ell}(\log(2g)+n)$.
The images of these loops by the inclusion map~$i$ yield the desired loops.
The second assumption follows from Theorem~\ref{jaco}.
\end{proof}

\medskip

\begin{remark}\label{remark.non.orientable}
A non-orientable version of Theorem  \ref{maintheorem1.with.lower.bound}
holds. Let~$M$ be a closed non-orientable surface of genus~$g\geq 1$
with homotopical systole at least~$\ell$ and area normalized to~$g$.
Let~$n\in \{1,\ldots,g\}$.
There are at least~$n$
loops~$\gamma_1,...\gamma_n$
based at the same point~$v$ in~$M$
whose homology classes in~$H_1(M,\Z_2)$ are independent  such that
for every~$i=1,...,n$, we have
$$\length(\gamma_i)\leq 24C'_{\ell}(\log(g)+n),$$
where~$C'_{\ell}=\frac{C}{\min\{1,\ell\}}$
for some positive constant~$C$.
Moreover, if~$n<g$ then~$\langle \gamma_1,...,\gamma_n \rangle$ is free of rank~$n$.
\end{remark}

\section{Cut loci and capturing the topology}\label{section.cut.loci}
In this section we extend the notion of cut locus defined
originally for points in a Riemannian manifold to simple closed
geodesics (this might be already defined but the author
didn't find a reference in the literature)
and we give some basic results for the new notion.
\\

Let~$M$ be a closed surface and~$p$ be a point in~$M$.
The cut point of~$p$ along a geodesic~$C_p$ starting at~$p$
is the first point~$q\in C_p$ such that the arc of~$C_p$
between~$p$ and any point~$r$
on~$C_p$ after~$q$ is no longer minimizing.
The set~$\Cut(p)$ of all cut points along all
the geodesics issued from~$p$ is called the cut locus of~$p$.
We extend this notion to simple closed geodesics as follows.
\\

Let~$\alpha:[0,l] \to M$ be a simple closed geodesic in~$M$
and~$\beta$ be another geodesic that starts
orthogonally from~$\alpha$ at some point~$p$.
The cut point of~$\alpha$
along~$\beta$ is the first point~$q\in \beta$ such
that, for any point~$r$ on~$\beta$ beyond~$q$ the length
of the arc of~$\beta$ between~$p$ and~$r$
no longer agrees with the distance from~$r$ to~$\alpha$.
The set~$\Cut(\alpha)$ of all the cut points of all the geodesics
issued orthogonally from~$\alpha$ is called the cut locus of~$\alpha$.
An alternative useful way to view~$\Cut(\alpha)$ is the
following. Denote by~$N\alpha$ the normal bundle to~$\alpha$.
Each vector~$v_t\in N\alpha$ gives
rise to a geodesic~$C_t$ starting at~$\alpha(t)$
such that~${C_t}'(0)=v_t$. Denote by~$q_t$
the cut point of~$\alpha$ along the geodesic~$C_t$.
The point~$q_t$ is the image by the exponential
map of some vector~$v'_t$ parallel to~$v_t$.
Let~$N_1$ be the set of the vectors~$v'_t$
and~$N_2$ be the set of the vectors~$\lambda v'_t$,
where~$\lambda\in [0,1)$.
Then,~$\Cut(\alpha)=\expo(N_1)$.

\begin{lemma}
$$M=\exp(N_1)\cup \exp(N_2),$$
where the union is disjoint.
\end{lemma}

\begin{proof}
Let~$x$ be a point in~$M$.
There exists a minimizing geodesic~${\sigma_x}^{-1}$
from~$x$ to~$\alpha$ parametrized by arc length
such that~$\length({\sigma_x}^{-1})=\dist(x,\alpha)$.
The geodesic~${\sigma_x}^{-1}$ hits~$\alpha$ orthogonaly
in a point~$\alpha(t)$ ($\textit{cf.}$ \cite{Docarmo}).
Since~$\sigma_x$ is  minimizing,
the point~$x$ is not after the cut point
of~$\alpha$ along~$\sigma_x$. That means that the
vector~$\dist(\alpha(t),x)\sigma'(0)\in N_1\cup N_2$.
Notice that~$x=~\exp(\dist(\alpha(t),x)\sigma'(0))$.
Thus,~$M=\exp(N_1)\cup \exp(N_2)$.
\\
Now let us prove that the union is disjoint.
Let~$y\in \exp(N_1) \cap \exp(N_2)$. Since~$y \in \exp(N_2)$,
there exists a minimizing geodesic~$\sigma_y:[0,\ell]\to M$
from~$\alpha$ to~$y$, parametrized by arc length such that
$\sigma_y$ is still minimizing for some time after~$y$ i.e.
there exists an~$\varepsilon >0$
such that~$\sigma_y: [0,\ell+\varepsilon]$ is a minimizing geodesic
from~$\alpha$ to~$\sigma_y(\ell+\varepsilon)$.
On the other hand,
since~$y\in\exp(N_1)$, there exists a minimizing geodesic~$\delta_y$
from~$\alpha$ to~$y$ parametrized by arc length such
that~$\delta_y$ is no longer minimizing after~$y$.
Let~$\phi$ be the curve defined by~$\phi(t)=\delta_y(t)$ if~$t\in [0,\ell]$,
and~$\phi(t)=\sigma_y(t)$ for~$t\in [\ell,\ell+\varepsilon]$.
Let~$0<\varepsilon'<\varepsilon$. There exists a minimizing
geodesic from~$\phi(\ell-\varepsilon')$ to~$\phi(\ell+\varepsilon')$
which is of length strictly less than the arc of~$\phi$ between these two
points since~$\phi$ is not smooth at~$\phi(\ell)$.
We conclude that~$\dist(\sigma_y(\ell+\varepsilon'),\alpha)$ is strictly less than
the length of~$\sigma_y$ between~$\sigma_y(0)$ and~$\sigma_y(\ell+\varepsilon')$.
Hence a contradiction. So the proof is finished.
\end{proof}

\begin{lemma}\label{cut.locus.deformation}
The set~$\Cut(\alpha)$ is a deformation retract of~$M\setminus{\{\alpha\}}$.
We will say that~$\Cut(\alpha)$ captures the topology of~$M\setminus{\{\alpha \}}$.
\end{lemma}

\begin{proof}
Let~$x$ be a point of~$M$ not in~$\alpha$ or~$\Cut(\alpha)$.
Denote by~${\sigma_x}$ the
unique minimizing geodesic from~$x$ to~$\alpha$.
Let~$x'$ be the cut point of~$\alpha$ along the geodesic~$\sigma_x$.
Clearly,~$x'\in \Cut(\alpha)$.
Now we can shrink~$M\setminus{\{\alpha\}}$ to~$\Cut(\alpha)$
by sliding each point~$x$ of~$M$ not in~$\alpha$ or~$\Cut(\alpha)$
to~$\Cut(\alpha)$ along the arc of the geodesic~$\sigma_x$ between~$x$ and~$x'$.
\end{proof}

\begin{proposition}\label{propo.cut.graph}
Let~$(M,g)$ be a closed real analytic Riemannian
surface and~$\alpha$ be a simple closed geodesic in~$M$.
Then~$\Cut(\alpha)$ is a finite  graph.
\end{proposition}

We omit the proof of Proposition \ref{propo.cut.graph}
since it is essentially the same proof
as in \cite{Myers} p.97.

\section{Short Homotopically Independent loops on Riemannian Surfaces}\label{section.proof.maintheorem1}
In this section we prove Theorem~\ref{maintheorem1}.
Before doing that, let us give some definitions  and  some 
independent propositions that will be useful to the rest of this section.

\begin{lemma}\label{lemma.free.rank}
Let~$F= \langle a,b \rangle$ be a free subgroup of
rank 2 of the fundamental group of a closed Riemannian manifold.
For every integer~$n\geq~1$,
the  subgroup~$H=\langle b,a^1ba^{-1},...,a^{n-1}ba^{-(n-1)} \rangle$
of~$F$ is  free  of rank~$n$. Moreover, if~$\length(a)=l_a$ and~$\length(b)=l_b$.
Then,
$$\sup\limits_{0\leq i \leq n-1}\length(a^iba^{-i})\leq 2(n-1)l_a+l_b.$$
\end{lemma}

\begin{proof}
Since the subgroup of a free group is free then~$H$ is free.
Next, we claim that the generator~$a^pba^{-p}$ is not an element
of the free subgroup~$G$ generated by
the elements~$a^qba^{-q}$ for~$q\in \{0,\ldots, n-1\}\setminus \{p\}$.
Indeed, a reduced word in~$G$ starts with~$a^q$ with~$q\neq p$.
So~$H$ is of rank~$n$. The length inequality is immediate.
\end{proof}

\begin{proposition}\label{proposition.maximal.cylinder}
Let~$(M,g)$ be a compact Riemannian cylinder.
Denote by~$\alpha$ and~$\beta$ the two
boundary components of~$M$. Suppose that
$$\length(\alpha)<1<\length(\beta).$$
Then there exists a  non-contractible simple loop~$\gamma$ in~$M$
of length 1 such that the systole of the cylinder~$R_{\gamma}$
bounded by~$\beta$ and~$\gamma$ is equal to 1.
\\
\indent 
In particular, the loop $\gamma$ is a systolic loop of $R_{\gamma}$.
\end{proposition}

\begin{proof}
Let $X=\{ \sigma$ simple non-contractible loop in $M$ such that $\sys(R_\sigma)=1\}$,
where by  $R_\sigma$ we mean the cylinder of boundary components $\beta$ and $\sigma$.
Clearly the set $X$ is non empty.
Let $\ell=\inf_{\sigma\in X} \length(\sigma)$
and  $\varepsilon$ be  a small positive constant. 
By the definition of the infimum, there exists a simple non-contractible loop~$\sigma_0$
such that~$\sys(R_{\sigma_0})=1$ with~$\ell \leq \length(\sigma_0)\leq \ell +\varepsilon$.
The systolic loop~$\gamma$ of~$R_{\sigma_0}$ is 
a simple non-contractible loop 
in~$M$. Moreover, we have~$R_{\gamma} \subset R_{\sigma_0}$. Thus 
$$1=\sys(R_{\sigma_0}) \leq \sys(R_{\gamma})\leq \length(\gamma)=1.$$
So~$\sys(R_{\gamma})=1$. This finishes the proof. 
\end{proof}

\bigskip

\noindent In the proof of Theorem~\ref{maintheorem1.orientable} below,
we will need the following definition.

\begin{definition}
Let~$M$ be a closed  Riemann surface of genus~$g$
(with possibly one disk removed).
It is well known that such a surface can be obtained from a polygon~$P$
(with possibly one disk removed)
by pairwise identifications of its sides where all the vertices
of~$P$ get identified to a single point on~$x$ of~$M$.
Such a polygon, will be called a normal representation of~$M$.
After identification, the  edges  of~$P$ give rise to~$2g$
simple loops (in case~$M$ is orientable)
or to~$g$ simple loops (in case~$M$ is non-orientable)
based at~$x$ and intersecting each other only  at~$x$.
Such  set of loops is called a \emph{canonical system of loops}.
\end{definition}

\noindent Now we prove Theorem~\ref{maintheorem1}.

\begin{theorem}\label{maintheorem1.orientable}
Let~$M$ be a closed orientable Riemannian surface of
genus~$g\geq 2$. There are at least~$n=\ceil{\log(2g)+1}$
homotopically independent
loops~$\alpha_1,\ldots ,\alpha_n$
based at the same point such that for
all~$i=1,\ldots,n$,
$$\length(\alpha_i)\leq 2^{20}\frac{\log(g)}{\sqrt{g}}\sqrt{\area(M)}.$$
\end{theorem}

\begin{proof}[Proof of Theorem \ref{maintheorem1.orientable}]
Since every smooth metric can be approximated by a real analytic one,
we can assume that~$M$ is a real analytic Riemannian surface.
Multiplying the metric by a constant if needed, we can suppose that
the area of~$M$ is normalized to~$g$.
We only need to consider the case where the homotopical
systole of~$M$ is less than~$1$, since the other case is settled 
down by Theorem \ref{maintheorem1.with.lower.bound}.
Consider a maximal set~$X$ of simple closed geodesics
$\alpha_1,...,\alpha_p$ of length at most~$1$ which
are pairwise disjoint in~$M$ and non freely homotopic.
Let~$k$ be the number of elements of~$X$ that are separating.
Note that~$k\leq p$.
The main idea of the proof is to go back to the case
where the homotopical systole is at least~$1$.

\begin{remark}
At first, we were tempted to cut the surface~$M$
open along  the loops~$\alpha_i$ of~$X$ and  to
attach an hemisphere along each of the~$2p$ boundary components.
This yields  at least~$k+1$ new closed surfaces~$M_1,\ldots,M_{k+1}$,
where~$k$ is the number of geodesics in~$X$ that are separating.
We hoped  to find the desired loops or two short
homotopically independent loops based at the
same point in one of the closed surfaces~$M_i$.
Recall that the  homotopical systole of each~$M_i$ is at least~1
so we can use Theorem~\ref{maintheorem1.with.lower.bound}.
Afterwards we wanted to show that these loops do not
cross the hemispheres and so lie in the original surface~$M$.
It doesn't take much time to realize that this idea is naive.
One can run into many problems.
Let's imagine the  case where~$p=g$ and
all of the geodesics~$\alpha_i$ are
non-separating like the surface in Figure 2.
In this case, the surface obtained by cutting~$M$ along
the loops~$\alpha_i$ and attaching
hemispheres is of genus~0 and so the proof collapses.
Instead we will cut~$M$ along each~$\alpha_i$,
chop off some ``maximal" cylinders
and then glue the boundary components back together to obtain
a new surface with systole bounded away from zero.
\begin{figure}[H]
\begin{center}
\includegraphics[height=25mm]{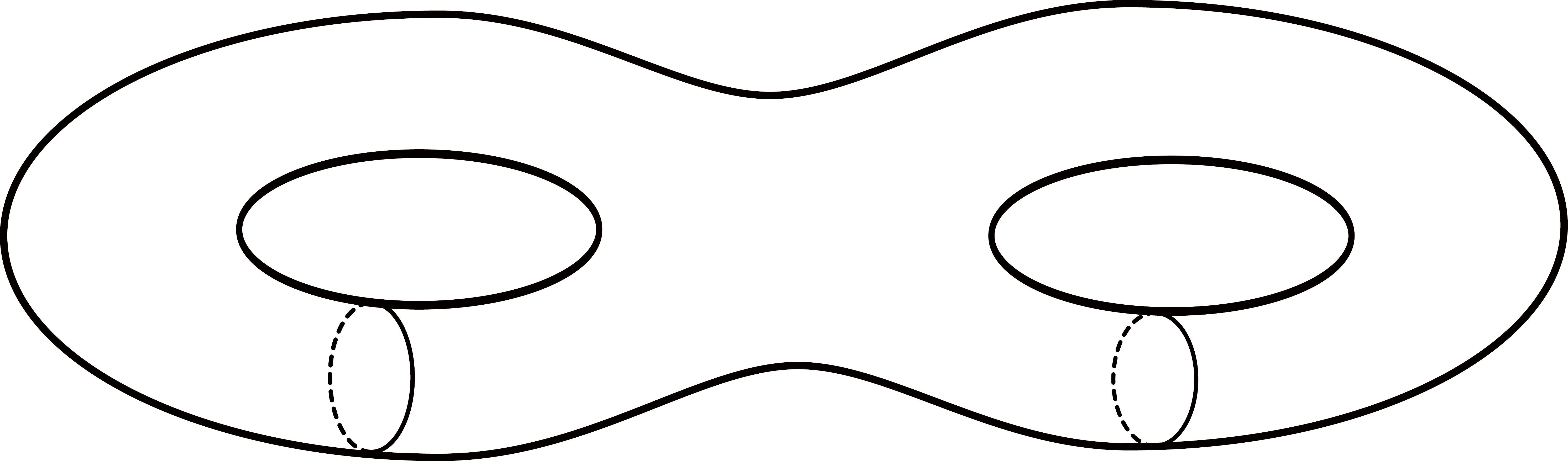}
\leavevmode
\put(-185,-3){\makebox(0,0){$\alpha_1$}}
\put(-60,-2){\makebox(0,0){$\alpha_2$}}
\caption{}
\end{center}
\end{figure}
\end{remark}

\noindent Let~$\varepsilon~\in~\{-,+\}$.
We divide the proof into 5 steps.
\\
\\
\noindent \emph{Step 1}. In this step we chop off
cylinders corresponding to short separating loops.
If~$k=~0$, we skip
this step and start directly at the second step.
By renumbering the~$\alpha_i$'s if needed,
we can suppose that for~$i~=~1,~\ldots,k$, the simple
closed geodesic~$\alpha_i$ is separating.
Cut the surface~$M$ open along~$\alpha_1$.
We obtain two compact surfaces~$M^-$
and~$M^+$ with signature~$(g-m,1)$
and~$(m,1)$, where~$m$ is some positive integer less than~$g$.
Denote by~$\alpha_1^\varepsilon$
the boundary of the surface~$M^\varepsilon$
and let~$S^\varepsilon$  be one of its canonical system of loops.
Notice that since the genus of~$M^\varepsilon$ is at least~1,
we have~$\card(S^\varepsilon)\geq 2$.
We can suppose that for every pair of loops~$a$ and~$b$ in~$S^{\varepsilon}$,
we have~$\sup(\length(a),\length(b))>1$.
Otherwise the proof is finished by  Lemma~\ref{lemma.free.rank}
since~$a$ and~$b$ do not commute and so generate a free group of rank~2.
Cut the surface~$M^\varepsilon$  open along the loops in~$S^\varepsilon$.
This gives rise to a cylinder~$T^\varepsilon$ with two boundary
components~$\alpha_1^{\varepsilon}$ and~$\beta_1^{\varepsilon}$
such that~$\length(\beta_1^{\varepsilon})>1$.
So the  cylinder~$T^\varepsilon$  satisfies the hypothesis of
Proposition~\ref{proposition.maximal.cylinder}.
Thus, there exists a non-contractible 
simple loop~$\gamma_1^\varepsilon$
of length 1 which is a systolic loop of the
cylinder~$R_1^\varepsilon$  bounded
by~$\beta_1^\varepsilon$ and~$\gamma_1^\varepsilon$
is 1. Cut~$T^\varepsilon$ along~$\gamma_1^\varepsilon$ and
throw away the cylinder~$C_1^\varepsilon$ bounded
by~$\alpha_1^\varepsilon$ and~$\gamma_1^\varepsilon$.
Now re-glue~$R_1^\varepsilon$
by pairwise identifications of the edges of~$\beta_1^\varepsilon$.
This gives rise to a compact surface~$M_1^\varepsilon$
with one boundary component~$\gamma_1^\varepsilon$  of length~1.
Glue the surfaces~$M_1^-$ and~$M_1^+$ along
their boundaries~$\gamma_1^-$ and~$\gamma_1^+$.
The resulting surface~$M_1$,  satisfies the following.
\begin{itemize}
\item The surface~$M_1$ has the same genus as the surface~$M$;
\item $\area(M_1) \leq \area(M)$;
\item A minimal representative in~$M_1$ of the
free homotopy class  of~$\alpha_1$ is given by the 
simple loop~$\gamma_1$ of length 1 
obtained by gluing~$\gamma_1^-$ and~$\gamma_1^+$ together.
\end{itemize}
Repeat the above  process with the~$k-1$ remaining
elements of~$X$ that are separating.
This gives rise to a closed surface~$M_k$ of the same genus as the surface~$M$
such that~$\area(M_k) \leq \area(M)$. Moreover, any simple closed geodesic
of~$M_k$ of length less than~$1$ is non-separating.
Perturbing the metric again, we can suppose again that it is a real analytic one.

\medskip

\noindent \emph{Step 2.} In this step, we chop off
cylinders corresponding to short non-separating loops.
Cut the surface~$M_k$ open along~$\alpha_{k+1}$.
This leads to a surface~$N_k$ with genus~$g-1$ and with
two boundary components~$\alpha_{k+1}^-$ and~$\alpha_{k+1}^+$.
By Lemma~\ref{cut.locus.deformation}, we know that the cut locus~$\Cut(\alpha_{k+1})$
of~$\alpha_{k+1}$ is a deformation retract of~$M\setminus \{\alpha_{k+1}\}$.
So the fundamental group of~$\Cut(\alpha_{k+1})$ is
isomorphic to the fundamental group of~$N_k$.
Now cut the surface~$N_k$ open along~$\Cut(\alpha_{k+1})$.
This gives rise to two cylinders.
The cylinder~$T_{k+1}^-$ with
boundary components~$(\alpha_{k+1}^-,\beta_{k+1}^-)$
and the cylinder~$T_{k+1}^+$ with boundary
components~$(\alpha_{k+1}^+,\beta_{k+1}^+)$.
Arguing as in Step 1, we can suppose
that~$\length(\beta_{k+1}^\varepsilon)>1$.
So the cylinder~$T_{k+1}^\varepsilon$ satisfies the hypothesis
of Proposition \ref{proposition.maximal.cylinder}.
Thus there exists a  non-contractible simple 
loop~$\gamma_{k+1}^\varepsilon$ of length 1
which is a systolic loop of the cylinder~$R_{k+1}^\varepsilon$
of boundary components~$(\beta_{k+1}^\varepsilon,\gamma_{k+1}^\varepsilon)$  is 1.
Cut~$T_{k+1}^\varepsilon$ open along~$\gamma_{k+1}^\varepsilon$
and throw away the cylinder~$C_{k+1}^\varepsilon$
bounded by~$\alpha_{k+1}^\varepsilon$ and~$\gamma_{k+1}^\varepsilon$.
Now re-glue the cylinder~$R_{k+1}^\varepsilon$ by re-identifying
the sides of~$\beta_{k+1}^\varepsilon$.
This gives rise to two compact surfaces~$M_{k+1}^-$ and~$M_{k+1}^+$
with boundary components that can be pairwise identified.
Gluing these two surfaces together we get a closed 
surface~$M_{k+1}$ that satisfies the following.
\begin{itemize}
\item The surface~$M_{k+1}$ has the same genus as the surface~$M_k$.
\item~$\area(M_{k+1}) \leq \area(M_{k})$.
\item A minimal representative of the free homotopy class
of~$\alpha_{k+1}$ in~$M_{k+1}$ is given by the simple loop~$\gamma_{k+1}$
of length 1, obtained by gluing~$\gamma_{k+1}^-$ and~$\gamma_{k+1}^+$ together.
\end{itemize}
Repeat the above  process with  the~$p-k-1$
remaining elements of~$X$.
This gives rise to a closed surface~$M_p$
of the same genus as the surface~$M$
such that~$\area(M_p) \leq \area(M)$.

\medskip

\noindent Before proceeding to the next step,
recall that the simple closed geodesics~$\alpha_1,\ldots,\alpha_p$ in the original surface~$M$
correspond to the simple closed geodesics~$\gamma_1,\ldots,\gamma_p$ in the surface~$M_p$.
Also recall that the cylinders~$C_i^-$ and~$C_i^+$ in~$M$ 
share the same boundary component~$\alpha_i$.
We denote by~$C_i$ the cylinder with boundary components~$(\gamma_i^-,\gamma_i^+)$,
that is,~$C_i= C_i^+\cup C_i^-$.

\medskip

\noindent \emph{Step 3}. In this step, we  show that
we can suppose that two different cylinders~$C_j$
and~$C_{j'}$ in~$M$ are distant from each other.
Specifically, we have~$\dist_M(C_j ,C_{j'})>2^{18}\log(g)$.
In other words, we have
\begin{eqnarray}
\dist_{M_p}(\gamma_j,\gamma_{j'})>2^{18}\log(g).
\end{eqnarray}
Indeed, suppose the opposite. Without loss of generality,
suppose that the distance between~$C_j$ and~$C_{j'}$
is equal to~$\dist(\gamma_j^-,\gamma_{j'}^-)$.
Let~$z_1$ be a point on~$C_j$ and~$z_2$ be a point on~$C_{j'}$
such that~$\dist(z_1,z_2)=\dist(\gamma_j^-,\gamma_{j'}^-)$.
Consider the loop~$\mu$ that starts at~$z_1$, travels along a
minimizing geodesic between~$z_1$ and~$z_2$,
makes a complete tour along~$\gamma_{j'}^-$ and then comes back to~$z_1$.
We have that~$\length(\mu) \leq 2^{19}\log(g)+1$.
Notice also that~$\mu$ and~$\gamma_j^-$ do not commute.
In particular, they are homotopically independent.
So by  Lemma~\ref{lemma.free.rank} (take~$a=\gamma_j^-$ and~$b=\mu$),
the proof of the theorem is finished.

\medskip

\noindent \emph{Step 4}. In this step, we show
that we can suppose that
$$\sys(M_p)~\geq~1.$$
Indeed, by contradiction, suppose that there is a systolic loop~$\mu$ 
of~$M_p$ of length less than~1.
We claim that the geodesic~$\mu$  transversally
intersects at least one of the~$\gamma_i's$.
Indeed, suppose the opposite, and
denote by~$\mu'$ the simple closed geodesic
in the original surface~$M$
that corresponds to~$\mu$.
Since~$\mu$ does not transversally
intersects any of the~$\gamma_i's$,
the loop~$\mu'$ is disjoint from all the cylinders~$C_i$.
In particular,~$\mu'$ does not intersect any of the loops~$\alpha_i$.
This contradicts the maximality of~$X$, 
since~$\length(\mu')<1$.
\indent Let~$j\in \{1,\ldots,n \}$ be such that~$\mu$
transversally intersects~$\gamma_j$.
That means that in the surface~$M$, the loop~$\mu'$
goes across the cylinder~$C_j$.
Now we claim that~$\mu$ intersects only one~$\gamma_j$.
Indeed, the length of~$\mu'$ is less than~$1$ and the distance
between any pair of cylinders~$C_j$ and~$C_{j'}$ is greater than~$1$.
Therefore,~$\mu$ intersects only one~$\gamma_j$.
Moreover, the two minimizing simple 
loops~$\mu$ and~$\gamma_j$ do not commute.

\medskip

\begin{lemma}\label{lemma.two.independent.short.loops}
Let~$\beta$ be a loop in~$M_p$ of length less than~$L$
that transversally intersects only one geodesic~$\gamma_j$ and does not commute with it.
Then there exist two loops~$a,b$ based at the same  point
in the original surface~$M$
that do not commute and such that~$\length_M(a)=1$
and~${\length}_M(b)\leq 2L+1$.
In particular, the loops~$a$ and~$b$ are homotopically independent.
\end{lemma}

\begin{proof}
We give~$\beta$ and~$\gamma_j$ some orientation.
Let~$x_1,\ldots,x_q$ be the transversal  intersection
points of~$\beta$ and~$\gamma_j$ counted with multiplicity and
ordered in the sense that if we start walking on~$\beta$,
then~$x_i$ is the~$i-th$ time~$\beta$ intersects~$\gamma_j$.
Suppose that~$q\geq 2$ (the case~$q=1$ will be treated in the end of the proof).
Let~$\beta_{i,i+1}$ be the simple loop based
at~$x_i$ defined as the concatenation of the
oriented arc of~$\beta$ between~$x_i$ and~$x_{i+1}$ and the oriented
arc~$c_{i+1,i}$ of~$\gamma_j$ between~$x_{i+1}$ and~$x_i$.
The loop~$\beta$ is homotopic to the loop
$\beta_{1,2}c_{1,2}\ldots\beta_{q,q+1}c_{q,q+1},$
where by convention~$c_{i,i+1}$ is the inverse of~$c_{i+1,i}$, and~$x_{q+1}=x_1$.
\\
\indent
Notice from the above equality that at least one
of the curves~$\beta_{i,i+1}c_{i,i+1}$ does not commute with~$\gamma_j$, for otherwise we
will have that~$\beta$ commute with~$\gamma_j$, which is a contradiction.
\\
\indent
Now let~$\beta_{k,k+1}c_{k,k+1}$ be one of the curves~$\beta_{i,i+1}c_{i,i+1}$
that does not commute  with~$\gamma_j$.
The curve~$\beta_{k,k+1}c_{k,k+1}$ is homotopic to~$\beta_{k,k+1}$,
so in particular~$\beta_{k,k+1}$ does not commute with~$\gamma_j$.
Recall that the surface~$M$ can be obtained from the surface~$M_p$
by cutting along the~$\gamma_i$'s and re-inserting the cylinders~$C_i$.
Thus, the loop in~$M$ that corresponds to~$\beta$ decomposes 
into a union  of curves whose endpoints  lie on one 
of the two boundary components~$\gamma_j^-$ 
and~$\gamma_j^+$ of the cylinder~$C_j$.
Denote by~$x'_k$ and~$x'_{k+1}$ the points in~$M$ corresponding
to the points~$x_k$ and~$x_{k+1}$ of~$\beta_{k,k+1}$ in~$M_p$.
We have two cases.
\\
\\
\emph{Case 1}. The points~$x'_k$ and~$x'_{k+1}$  lie both the 
same boundary component, say~$\gamma_j^+$.
\\ In this case, let~$\beta'$ be the simple loop in~$M$ that corresponds to~$\beta_{k,k+1}$
(See Figure 3).

\begin{minipage}{\linewidth}
\begin{center}
\includegraphics[keepaspectratio=true,scale=0.2]{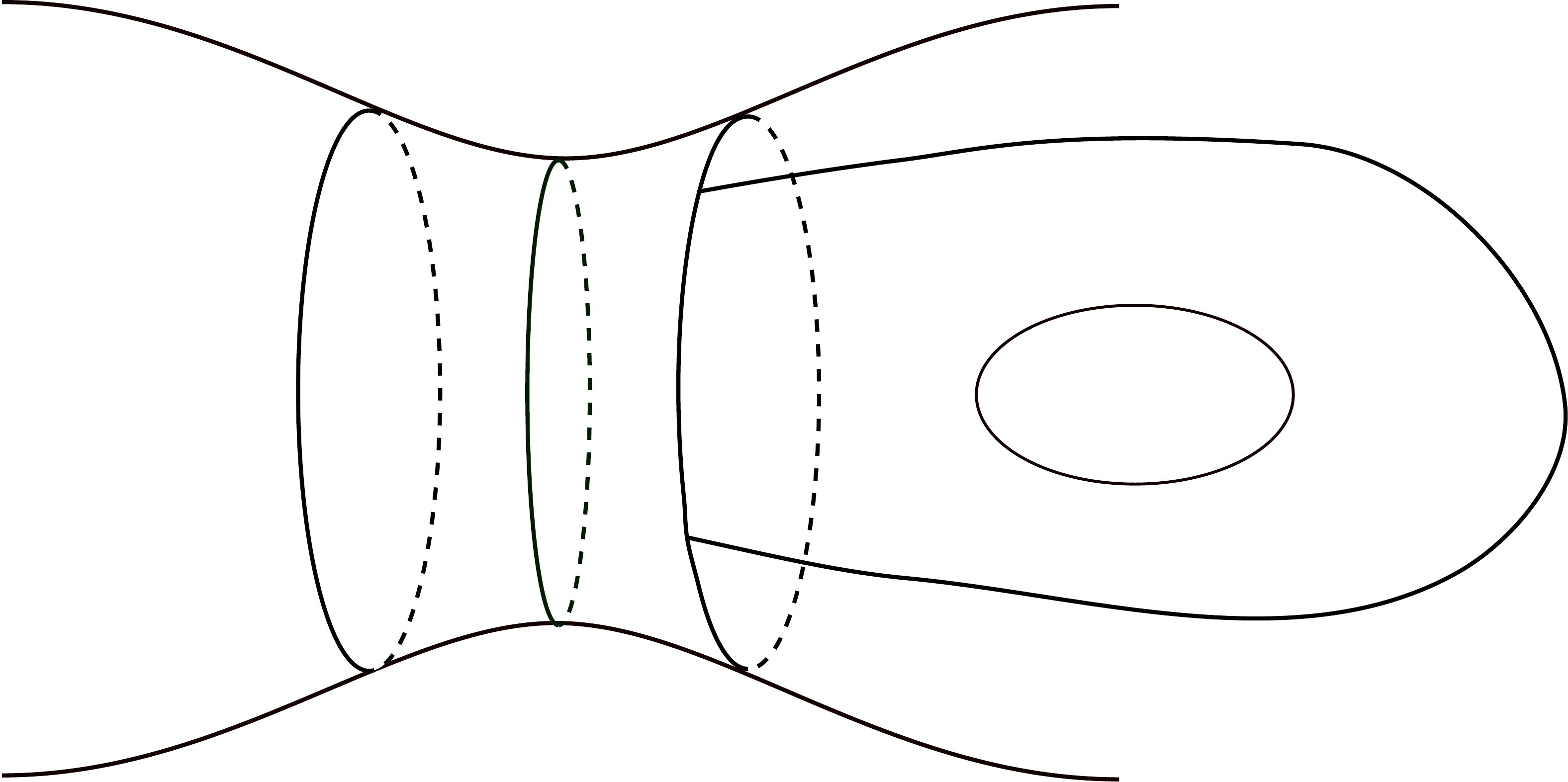}
\put(10,80){\makebox(0,0){$M$}}
\put(-10,50){\makebox(0,0){$\beta'$}}
\put(-115,10){\makebox(0,0){$\alpha_j$}}
\put(-135,5){\makebox(0,0){$\gamma_j^-$}}
\put(-100,3.5){\makebox(0,0){$\gamma_j^+$}}
\captionof{figure}{}
\end{center}
\end{minipage}

\noindent Take~$a=\gamma_j^+$ and~$b= \beta'$.
These two loops are based at the same point and do not commute.
Moreover we have~$\length(a)=1$ 
and~$\length(b)\leq L+1$.
\\
\\
\emph{Case 2}. The points~$x'_k$ and~$x'_{k+1}$ do not lie both on~$\gamma_j^-$ or~$\gamma_j^+$.
\\In this case, let~$\beta'$ be the arc in~$M$ that corresponds
to the  arc of~$\beta$ between~$x_k$ and~$x_{k+1}$.
\\
\\
\begin{minipage}{\linewidth}
\begin{center}
\includegraphics[keepaspectratio=true,scale=0.2]{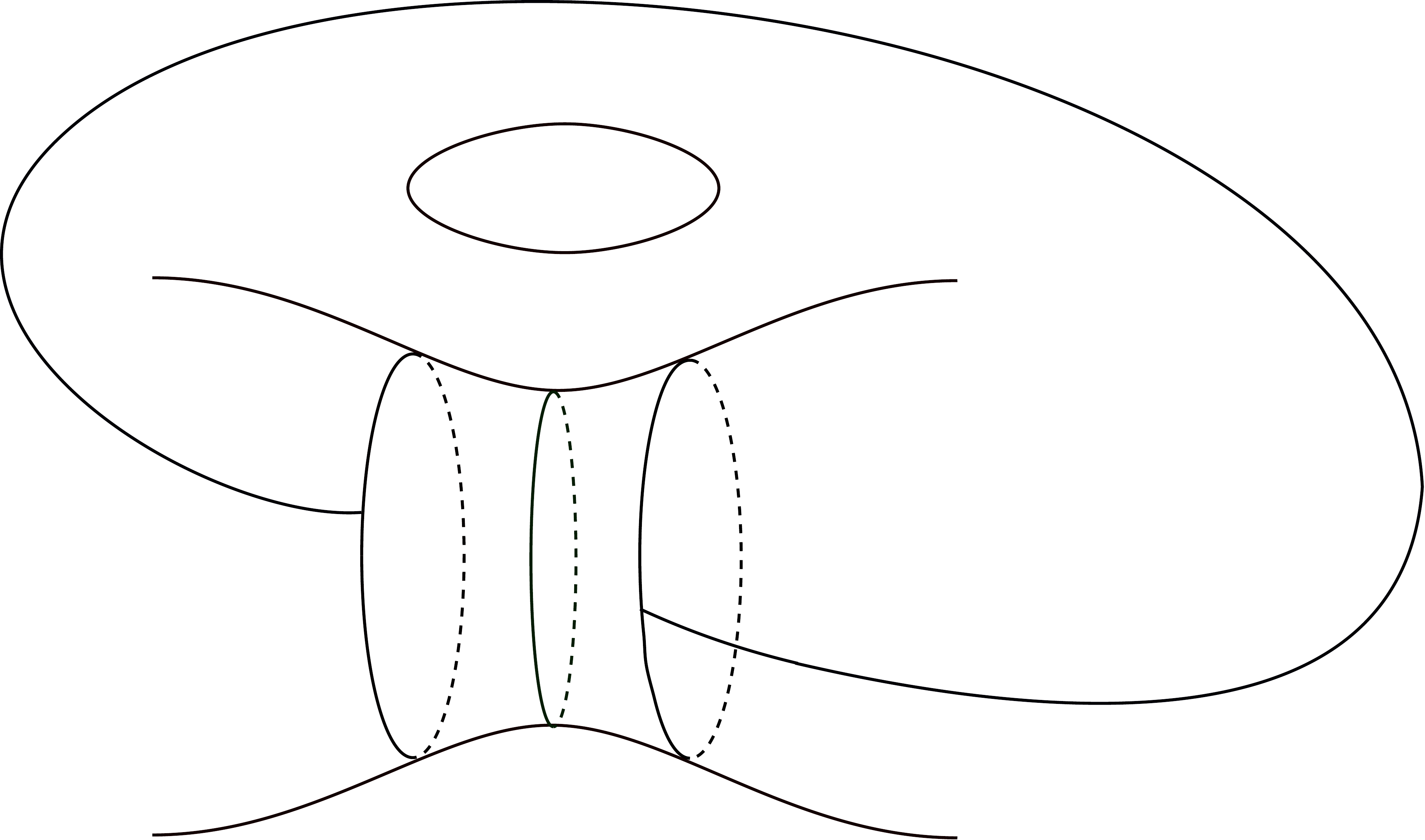}
\put(15,100){\makebox(0,0){$M$}}
\put(-10,50){\makebox(0,0){$\beta'$}}
\put(-140,10){\makebox(0,0){$\alpha_j$}}
\put(-165,3){\makebox(0,0){$\gamma_j^-$}}
\put(-120,3.5){\makebox(0,0){$\gamma_j^+$}}
\captionof{figure}{}
\end{center}
\end{minipage}
\noindent Take~$a=\gamma_j^+$ and~$b= \beta'\gamma_j^+\beta'^{-1}$.
These two loops are based at the same point and do not commute.
Moreover we have~$\length(a)=1$ and~$\length(b)\leq 2L+1$.
\\

Finally, if the number of intersections~$q=1$, we argue 
exactly like in case 2 above, supposing that~$x_{k+1}=x_k$.
That finishes the proof of the Lemma.
\end{proof}

Now, apply Lemma~\ref{lemma.two.independent.short.loops} with $\beta=\mu$
and make use of Lemma~\ref{lemma.free.rank} to finish the proof.

\medskip

\noindent \emph{Step 5}. By Theorem \ref{maintheorem1.with.lower.bound},
there are at least~$n=\ceil{\log(2g)+1}$ homotopically independent
geodesic loops~$\mu_1\,\ldots,\mu_n$
based at the same point in~$M_p$ with
$$\length(\mu_i)\leq 2^{18}\log(g).$$
If these loops are in the original surface~$M$,~$\ie$,  they don't
transversally  intersect any of the loops~$\gamma_i$ in~$M_p$,
then the proof is finished.
So suppose the opposite.
Let~$\mu$ be one the loops~$\mu_1\,\ldots,\mu_n$ that transversally  intersects
at least one of the~$\gamma_i$'s in~$M_p$.
From~(6.1), the loop~$\mu$ (transversally) intersects
exactly one loop~$\gamma_j$ in~$M_p$.
By Lemma~\ref{lemma.two.independent.short.loops},
we show that there exist two loops~$a,b$
in the original surface~$M$ based
at the same point with~$\length(a)=1$
and~$\length(b)\leq 2^{19}\log(g)+1$.
The result follows from Lemma~\ref{lemma.free.rank}.
\end{proof}

\begin{remark}
Theorem \ref{maintheorem1.orientable} extends to non-orientable surfaces
with multiplicative constant~$2^{22}$ instead of~$2^{20}$
by passing to the double oriented cover.
\end{remark}

\begin{corollary}
There exists a positive constant~$C$ such that the separating systole
of every closed Riemannian surface~$M$ of genus~$g\geq 2$ and area~$g$
satisfies
$$\sys_0(M)\leq C\log(g).$$
\end{corollary}

\begin{proof}
From Theorem~\ref{maintheorem1.orientable}, there exist two
non-commuting loops~$a$ and~$b$ based at the same point of length at
most~$c\log(g)$ for some positive constant~$c$.
The commutator~$[a,b]$ of~$a$ and~$b$,
of length at most~$4c\log(g)$,
yields a bound on the separating systole of~$M$.
\end{proof}

\end{document}